\documentclass{amsa}

\usepackage{graphicx,color} 

\usepackage{here}
\usepackage{amsmath, amsfonts, amssymb}
\usepackage{xcolor}
\usepackage{bigints}
\usepackage{enumerate}

\usepackage{indentfirst}
\usepackage{mathrsfs}

\numberwithin{equation}{section}
\numberwithin{figure}{section}
\numberwithin{table}{section}

\usepackage{amsthm}
\newtheoremstyle{mystyle}{}{}{\itshape}{}{}{}{}{}
\theoremstyle{plain}
\newtheorem{definition}{Definition}[section]
\newtheorem{theorem}{Theorem}[section]
\newtheorem{lemma}{Lemma}[section]
\newtheorem{proposition}{Proposition}[section]

\usepackage{tikz}

\graphicspath{{images/}}

\allowdisplaybreaks

\everymath{\displaystyle}

\begin{document}
\label{page:t}
\thispagestyle{plain}

\title{SOLVABILITY OF THE MOISTURE TRANSPORT MODEL FOR POROUS MATERIALS}
\author{Akiko Morimura\footnotemark[1]}
\affiliation{Department of Mathematical and Physical Science, Graduate School of Science,\\Japan Women's University\\
2-8-1 Mejirodai, Bunkyoku, Tokyo, 112-8681, Japan}
\email{m1716096ma@ug.jwu.ac.jp}
\sauthor{Toyohiko Aiki} 
\saffiliation{Department of Mathematical and Physical Science, Faculty of Science,\\Japan Women's University\\
2-8-1 Mejirodai, Bunkyoku, Tokyo, 112-8681, Japan}
\semail{aikit@fc.jwu.ac.jp}
\date{}

\footcomment{
This work is supported by JSPS KAKENHI Grant Number 22K03377 and partially supported Ebara Corporation.}
\footnotetext[1]{Corresponding author}
\maketitle
\vspace{-0.5cm}
\noindent
{\bf Abstract.}
We consider an initial and boundary value problem invoked from the mathematical model for moisture transport in porous materials. Because of the difficulty appearing in the boundary condition, we have changed it and obtain the nonlinear parabolic equation with the nonlinear boundary condition in the one-dimensional interval. The main result of this paper is to prove existence and uniqueness of solutions to the problem by applying the standard fixed point theorem argument.\\


\newpage
\section{Introduction}
For moisture transport in porous materials, a lot of mathematical models have been investigated. In this paper we focus on the model proposed by Fukui, Iba, Hokoi, and Ogura \cite{Fukui} whose motivation is based on \cite{Green}. Here, we introduce the model which describes the water transport with in a brick, and was proposed as a nonlinear diffusion equation on the one-dimensional interval in \cite{Fukui}. Let $\psi_w$ and $\phi$ be the water content and the porosity, respectively. Namely, the air content $\psi_a$ satisfies $\psi_w + \psi_a = \phi$. Accordingly, the masses $m_w$ and $m_a$ of water and air are given as $m_w = \rho_w \psi_w$ and $m_a = \rho_a \psi_a$, where $\rho_w$ and $\rho_a$ are densities of water and air, respectively. Under these notations, the mass conservation law for water and air implies the following two equations:
\begin{align}
\label{Fukui-eq-w}\frac{\partial m_w}{\partial t} 
&= \frac{\partial}{\partial x} \left( \lambda_p \frac{\partial P_w}{\partial x} \right) \;\; {\rm on} \; Q(T) := (0,T) \times (0,1), \\
\label{Fukui-eq-a}\frac{\partial m_a}{\partial t} 
&= \frac{\partial}{\partial x} \left( \frac{k_a}{g} \frac{\partial P_a}{\partial x} \right) \;\; {\rm on} \; Q(T),
\end{align}
where $T>0$, $x \in (0, 1)$ is the position in the material, $P_w$ and $P_a$ are pressures of water and air, respectively, and  $\lambda_p$ is the water conductivity, $k_a$ is the air permeability and $g$ indicates the gravitational acceleration. In \cite{Fukui}, the capillary pressure $P_c$ is defined by $P_c = P_a - P_w$ and assumed as $P_c = - \rho_w \mu$, where $\mu$ is the water chemical potential, and $\lambda_p$ is given by $\lambda_p = \frac{D}{\rho_w} \frac{\partial \psi_w}{\partial \mu}$, where $D$ is a function of $\psi_w$. Accordingly, we have
\begin{align}
\lambda_p \frac{\partial P_w}{\partial x} 
&= \frac{D(\psi_w)}{\rho_w} \frac{\partial \psi_w}{\partial \mu} \frac{\partial}{\partial x} \left( \rho_w \mu + P_a \right) \nonumber \\
&\label{Fukui-eq-w-lam}= D(\psi_w) \frac{\partial \psi_w}{\partial \mu} \frac{\partial}{\partial x} \left( \mu + \frac{P_a}{\rho_w} \right),
\end{align}
since $\rho_w$ is a constant. Here, we note that $\psi_w$ is defined as the function of $\mu$ in \cite{Fukui}.\\

Now, by putting $\lambda(\mu) = D(\psi_w(\mu)) \frac{\partial \psi_w}{\partial \mu}(\mu), p = \frac{P_a}{\rho_w}, \psi(\mu) = \rho_w \psi_w(\mu) = m_w$, and $u = \mu$, we obtain
\begin{align}
&\label{P-eq}\frac{\partial \psi(u)}{\partial t}  = \frac{\partial}{\partial x}\left( \lambda (u) \frac{\partial}{\partial x}(u+p) \right) \; \mathrm{in} \; Q(T).
\end{align}
In this paper, as a first step of research to the system consisting of \eqref{Fukui-eq-w} and \eqref{Fukui-eq-a}, we suppose that the function $p$ is given in $Q(T)$. In other words, we consider only \eqref{Fukui-eq-w}. We give a remark concerned with the relationship between \eqref{P-eq} and the standard porous media equation. To do so, by \eqref{Fukui-eq-w-lam} we can rewrite \eqref{P-eq} to the following form:
\begin{align}
&\label{P-eq-rw}\frac{\partial \psi(u)}{\partial t}  
= \frac{\partial}{\partial x}\left( D(\psi_w) \left( \frac{\partial \psi_w(\mu)}{\partial \mu} + \frac{\partial \psi_w}{\partial \mu} P_a \right) \right) \;\; {\rm in} \; Q(T).
\end{align}
In \cite{Fukui}, the function $D(\psi_w)$ is given by
\[
D(\psi_w) = 30.332 \times 10^{-6} \exp (79.8 \times \psi_w^{1.5}),
\]
(see Figure \ref{fig1}).
\begin{figure}[H]
\begin{minipage}[t]{0.5\linewidth}
\centering
\includegraphics[bb=0 0 560 420, width=80mm]{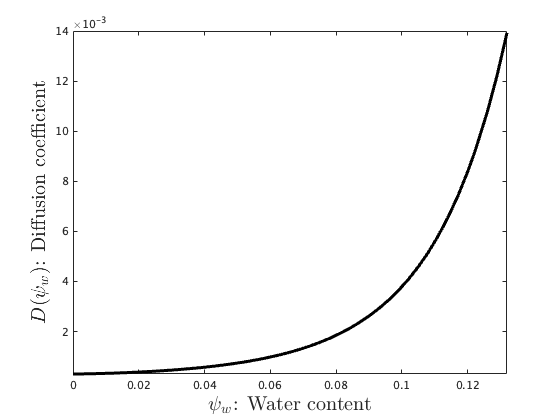} 
\caption{Graph of $D(\psi_w)$}
\label{fig1}
\end{minipage}
\begin{minipage}[t]{0.5\linewidth}
\centering
\includegraphics[bb=0 0 560 420, width=80mm]{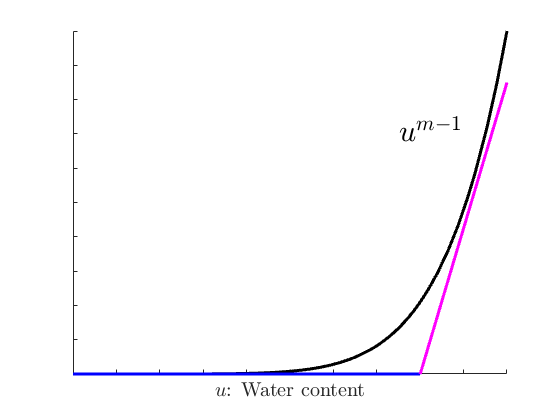} 
\caption{Typical graphs for porous media}
\label{fig2}
\end{minipage}
\end{figure}
It is well known that the water conductivity $D = D(u)$ is porous media is very small for small $u$ and increases rapidly for $u$ over some threshold as the graphs in Figure \ref{fig2}. Therefore, our equation \eqref{P-eq} can be regarded a kind of porous media equations.\\

Next, we mention our boundary condition. In \cite{Fukui}, since the water adsorption phenomenon from the top of the brick to the inside was discussed, the Robin boundary condition was imposed. In the present paper, for simplicity we assume no flux at the boundary, namely,
\begin{align*}
\lambda(u)\frac{\partial}{\partial x} (u + p) = 0 \;\; {\rm at} \; x =0,1.
\end{align*}
If $\lambda$ is strictly positive, then we have
\begin{align}
\label{P-BC-ex} \frac{\partial u}{\partial x} + \frac{\partial p}{\partial x} = 0 \;\; {\rm at} \; x =0,1.
\end{align}
The boundary condition \eqref{P-BC-ex} seems to be the Neumann type. However,  by the transformation to solve \eqref{P-eq}, \eqref{P-BC-ex} is rewritten to the non-monotone Robin type. Actually, by putting $v = \hat{\lambda}(u)$, we see that
\begin{align}
\frac{d \tilde{\lambda}(v)}{dv} \frac{\partial v}{\partial x} + \frac{\partial p}{\partial x} 
&= 0, \nonumber\\
\label{AP-BC-rw}\frac{\partial v}{\partial x} + a(v) \frac{\partial p}{\partial x} 
&= 0,
\end{align}
where $\hat{\lambda}(v)$ is the primitive of $\lambda$, $\tilde{\lambda}(v)$ is the inverse of $\hat{\lambda}(u)$ and $a(v) = \frac{1}{\frac{d \tilde{\lambda}(v)}{dv}}$ for $v \in \mathbb{R}$. Since $p$ is unknown in the original system, we can not assume the sign condition for $\frac{\partial p}{\partial x}$, that is, \eqref{AP-BC-rw} is not monotone with respect to $v$. Due to the evolution equation theory, we do not expect that nonlinear parabolic equation \eqref{P-BC-ex} with the boundary condition \eqref{AP-BC-rw} has a strong solution. Hence, in the present paper we shall show the strong solvability of the problem consisting of \eqref{P-eq} and the following modified boundary condition \eqref{P-BC}:
\begin{align}
\label{P-BC} \lambda(u) \frac{\partial u}{\partial x} + \frac{\partial p}{\partial x} = 0 \;\; {\rm at} \; x =0,1.
\end{align}
Also, we impose the initial condition,
\begin{align}
&\label{P-IC}u(0,x) = u_0(x) \; \mathrm{for} \; x \in(0,1).
\end{align}
Here, we note that the typical example of $\psi_w$ is the following:
\[
\psi_w(\mu) = \frac{0.0505}{8 + \exp(\log_{10}(-\mu) - 2)} + \frac{0.139}{1.1 + \exp(2.3\log_{10}(-\mu) - 4.6)},
\]
(see Figure \ref{fig3}).
\begin{figure}[h]
\centering
\includegraphics[bb=0 0 560 420, width=80mm]{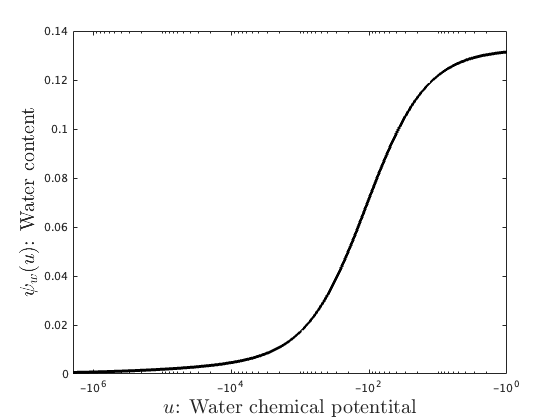} 
\caption{Graph of $\psi_w(u)$}
\label{fig3}
\end{figure}
Accordingly, the range of $\psi_w$ is bounded in $\mathbb{R}$. For this type of $\psi_w$, the equation \eqref{P-eq} is called elliptic-parabolic type and investigated. For instance, Alt-Luckhaus \cite{A-L} discussed the problem in case $p = 0$ and $\lambda(u) = \lambda(\psi(u))$, and Kenmochi-Pawlow studied in case $p = 0$. Recently, Uchida \cite{Uchida} considered the problem with $p = 0$ and a maximal monotone graph $\psi_w$. The aim of this paper is to treat the problem with $p$. However, the existence of $p$ leads to the difficulty for analysis. Therefore, we assume the following conditions to $\psi$ and $\lambda$:\\
\quad(A1) $\psi \in C^2({\mathbb R}), \psi>0, \delta_\psi \leq \psi' \leq C_{\psi}, |\psi''| \leq C_{\psi}$ on $\mathbb R$;\\
\quad(A2) $\lambda \in C^2({\mathbb R}), \delta_{\lambda} \leq \lambda \leq C_{\lambda}, |\lambda'|, |\lambda''| \leq  C_{\lambda}$ on $\mathbb R$.\\
where $\delta_\psi, C_\psi, \delta_\lambda$ and $C_\lambda$ are positive constants.

Under these assumptions we aim to establish existence and uniqueness of a strong solution to the problem P$\;:= \{ \eqref{P-eq}, \eqref{P-BC}, \eqref{P-IC} \}$.

\section{Notation and main result}
Throughout this paper, we use the following notations: 
\begin{align*}
H := L^{2}(0, 1), \quad X := H^{1}(0, 1),
\end{align*}
with the standard norms of $H, X$ and the inner product of $H$ denoted by $|\cdot|_{H} := |\cdot|_{L^{2}(0, 1)}$, $|\cdot|_{X} := |\cdot|_{H^{1}(0, 1)}$, and $(\cdot, \cdot)_H$, respectively. Also, let $D(f)$ be the effective domain and $\partial f$ be the subdifferential of a convex function $f$ on $H$.\\


Moreover, we define a strong solution of P on $Q(T)$ and give the main result of this paper as follows.
\begin{definition}
A function $u$ from $Q(T)$ to $\mathbb{R}$ is called a strong solution of P on $Q(T)$ if $u \in W^{1,2}(0, T; H) \cap L^\infty (0, T; X) \cap L^2 (0, T; H^2(0, 1))$ and satisfies 
\eqref{P-eq}, \eqref{P-BC} and \eqref{P-IC} in the usual sense.
\end{definition}

\begin{theorem}\label{P-SS}
Let $T>0$. If (A1), (A2), $p \in W^{1,2} (0, T; H^2(0,1)) $ and $u_0 \in X$ hold, then P has a unique strong solution on $Q(T)$.
\end{theorem}

We prepare fundamental results to prove the main theorem. In Section 4 we prove the theorem. 

\section{Preparation for the proof}
First, we give some useful inequalities without proofs.
\begin{lemma}\label{abs_ineq}
For any $u \in X$, it holds that
\begin{align}
\label{abs_L2partition} |u(x)|^2 
&\leq |u|_H^2 + 2|u|_H |u_x|_H \;\; {\it for} \; x \in [0,1],\\
\label{abs_L4partition}|u|^4_{L^4(0,1)} 
&\leq (|u|_H^2 + 2|u|_H |u_x|_H)^2.
\end{align}
\end{lemma}
In order to deal the nonlinear diffusion term, we introduce the function $\hat{\lambda}$ on $\mathbb{R}$ by $\hat{\lambda}(u) = \int^u_0 \lambda(r) dr$ for any $u \in \mathbb{R}$. For $\hat{\lambda}, \lambda$ and $\psi$, we can easily obtain
\begin{align}
\label{func_ineq}\delta_\lambda |u| 
&\leq |\hat{\lambda}(u)| \leq C_\lambda |u| \; {\rm for} \; u \in \mathbb{R}.
\end{align}

Next, we denote AP($\tilde{u}$) by the auxiliary problem \eqref{AP-eq} - \eqref{AP-IC} for any a given function $\tilde{u}$ on $Q(T)$:
\begin{align}
&\label{AP-eq}\frac{\partial}{\partial t} \psi(u) = \frac{\partial}{\partial x}\left( \lambda (u) \frac{\partial u}{\partial x} + \lambda (\tilde{u}) \frac{\partial p}{\partial x} \right) \; \mathrm{in} \; Q(T),\\
&\lambda(u)\frac{\partial u}{\partial x} + \frac{\partial p}{\partial x} = 0 \; \mathrm{at} \; x =0,1,\\
&\label{AP-IC}u(0,x) = u_0(x) \; \mathrm{for} \; x \in(0,1).
\end{align}

We prove Theorem \ref{P-SS} by applying the Banach fixed point theorem for a solution operator obtained from AP($\tilde{u}$). To deal with it, we define a solution AP($\tilde{u}$) and give a proposition concerned with solvability of AP($\tilde{u}$).

\begin{definition}
For any function $\tilde{u}$ in $Q(T)$, a function $u$ from $Q(T)$ to $\mathbb{R}$ is called a strong solution of AP($\tilde{u}$) on $Q(T)$ if $u \in W^{1,2}(0, T; H) \cap L^\infty (0, T; X) \cap L^2 (0, T; H^2(0, 1))$ satisfies \eqref{AP-eq} - \eqref{AP-IC} in the usual sense.
\end{definition}

\begin{proposition}\label{AP-SS}
Let $T>0, p \in W^{1,2} (0, T; H^2(0,1))$ and $u_0 \in X$. If $\tilde{u} \in L^2(0,T; X)$, then AP($\tilde{u}$) has a unique strong solution on $Q(T)$.
\end{proposition}

To prove Proposition \ref{AP-SS}, we  rewrite AP($\tilde{u}$) as follows. Put $v = \hat{\lambda}(u)$. Then, for given $\tilde{u}$ it holds that
\begin{align}
&\label{APbar-eq}\frac{\partial b(v)}{\partial t} = \frac{\partial^2 v}{\partial x^2} + f \; \mathrm{in} \; Q(T),\\
&\frac{\partial v}{\partial x} + h = 0 \; \mathrm{at} \; x =0,1,\\
&\label{APbar-IC}v(0,x) = v_0(x) \; \mathrm{for} \; x \in(0,1),
\end{align}
where $b = \psi(\hat{\lambda}^{-1})$ on $\mathbb{R}$, $f =\frac{\partial}{\partial x} \left\{ \lambda(\tilde{u}) \frac{\partial p}{\partial x} \right\}$ and $h  = \frac{\partial p}{\partial x}$ on $Q(T)$, and $v = \hat{\lambda}(u), v_0 = \hat{\lambda}(u_0)$ on $(0, 1)$. We denote the problem \eqref{APbar-eq} - \eqref{APbar-IC} by $\overline{\rm AP}$($\tilde{u}$). In order to apply the evolution equation theory, we put $B = b^{-1}$ and define a function $\varphi^t$ from $H$ to $( -\infty, \infty]$ by
\begin{align*}
\varphi^t (z) = \left\{
\begin{array}{ll}
\frac{1}{2} \int^1_0 \left( \frac{\partial z}{\partial x} \right)^2 dx + h(t, 1)z(1) - h(t, 0)z(0) & {\rm if} \; z \in X,\\
\infty & {\rm otherwise}.
\end{array}
\right.
\end{align*}
For $\varphi^t(z)$, the following lemma holds. 

\begin{lemma}\label{ap_lem_first}
If $h(\cdot, i) \in W^{1,2}(0, T)$ for $i = 0,1$, then $\varphi^t(z)$ is proper, convex and lower semi continuous (l.s.c.) on $H$.
\end{lemma}
\begin{proof}
It is obvious that $\varphi^t$ is convex and proper on $H$ for any $t \in [0 ,T]$.

Let $t \in [0,T]$. In order to prove that $\varphi^t$ is l.s.c. we show that $\varphi^t(z) \leq \liminf_{n \rightarrow \infty} \varphi^t(z_n)$ for any sequence $\{ z_n \}$ of $H$ with $z_n \to z$ in $H$ as $n \to \infty$. Assume $\liminf_{n \rightarrow \infty}\varphi^t(z_n) < \infty$, and put $\alpha = \liminf_{n \rightarrow \infty}\varphi^t(z_n)$. It is easy to see that there exists a subsequence $\{n_j\}$ that satisfies $\varphi^t(z_{n_j}) \rightarrow \alpha$ as $j \to \infty$. From now on, we write $z_j = z_{n_j}$ for each $j$. Since $\{ z_j \}$ and $\{ \varphi^t(z_j)\}$ are convergence sequences, there exists a constant $C>0$ such that $|z_j|_H \leq C$ and $\varphi^t(z_j) \leq C$ for any $j \in \mathbb{N}$. Therefore, by using Lemma \ref{abs_ineq} and Young's inequality, we have
\begin{align}
C &\geq \frac{1}{2} \int^1_0 \left(\frac{\partial z_j}{\partial x} \right)^2 dx - (|h(t,1)| + |h(t,0)|)\left(2|z_j|_H\left|\frac{\partial z_j}{\partial x}\right|_H + |z_j|^2_H\right)^{\frac{1}{2}}\nonumber \\ 
&\label{lem33-eq} \geq \frac{1}{2} \int^1_0 \left(\frac{\partial z_j}{\partial x} \right)^2 dx - \frac{1}{2}(|h(t,1)| + |h(t,0)|)^2 - \frac{1}{2} \left(3|z_j|^2_H + \frac{1}{2}\left|\frac{\partial z_j}{\partial x}\right|^2_H\right).
\end{align}
This inequality implies that $\{z_j\}$ is bounded in $X$.

Accordingly, we can take a subsequence $\{ j_k\}$ and $\hat{z} \in X$ such that
\[
z_{j_k} \rightarrow \hat{z} \; \mathrm{weakly} \; \mathrm{in} \; X \; {\rm as} \; k \to \infty.
\]
Clearly, we have $z = \hat{z}, z_{j_k}(i) \to z(i)$ in $\mathbb{R}$ as $k \to \infty$ for $i = 0, 1$, and \\$\liminf_{k \to \infty} \frac{1}{2} \int^1_0 \left| \frac{\partial z_{j_k}}{\partial x} \right|^2 dx \geq \frac{1}{2} \int^1_0 \left| \frac{\partial z}{\partial x} \right|^2 dx$. Thus, $\varphi^t$ is l.s.c. on $H$.
\end{proof}

We show the following  lemmas concerned with the subdifferential of $\varphi^t$.

\begin{lemma}
Let $t \in [0, T]$ and $z \in H$. If $\partial\varphi^t(z) \neq \varphi$, then $\partial \varphi^t(z)$ is single-valued.
\end{lemma}
\begin{proof}
By the assumption we have $z \in X$.
Let $z_1^\ast$ and $z_2^\ast \in \partial \varphi^t(z)$. From the definition of subdifferential, it follows that
\[
(z_i^\ast, y-z)_H \leq \varphi^t(y) - \varphi^t(z) \;\; {\rm for} \; i = 1,2, \; y \in H.
\]
Let $\eta \in X$ and $\varepsilon_1 > 0$. Since $z +\varepsilon_1 \eta \in X$, it follows that
\[
(z_i^\ast,  \eta)_H 
\leq \int^1_0 \frac{\partial z}{\partial x}\frac{\partial \eta}{\partial x} dx + \frac{\varepsilon_1}{2}\int^1_0 \left(\frac{\partial \eta}{\partial x}\right)^2 dx + h(t,1)\eta(1) - h(t,0)\eta(0) \;\; {\rm for} \; i = 1,2.
\]
Here, by letting $\varepsilon_1 \to 0$,
\begin{align}
\label{lem2_leq}(z_i^\ast,  \eta)_H 
\leq \int^1_0 \frac{\partial z}{\partial x}\frac{\partial \eta}{\partial x} dx + h(t,1)\eta(1) - h(t,0)\eta(0) \;\; {\rm for} \; i = 1,2.
\end{align}
By substituting $-\eta \in X$ into \eqref{lem2_leq}, we have 
\begin{align}
\label{b2}(z_i^\ast,  \eta)_H = \int^1_0 \frac{\partial z}{\partial x}\frac{\partial \eta}{\partial x} dx + h(t,1)\eta(1) - h(t,0)\eta(0) \;\; {\rm for} \; i = 1,2, \;  \eta \in X.
\end{align}
Therefore, we get $(z_1^\ast - z_2^\ast, \eta)_H = 0$ for any $\eta \in X$. Thus, we have $z_1^\ast = z_2^\ast \; \mathrm{in} \; H$.
\end{proof}

\begin{lemma}\label{b1-5}
For any $t \in [0, T]$ and $z \in H$, $z^\ast = \partial \varphi^t(z)$ if and only if the following (B1)-(B4) hold. 
\begin{enumerate}[\textrm{(B}1\textrm{)}]
\item $z \in W^{2,2}(0,1), z^\ast \in H$.
\item $(z^\ast, \eta)_H = \int^1_0  \frac{\partial \eta}{\partial x} \frac{\partial z}{\partial x} dx + h(t, 1)\eta(1) - h(t, 0)\eta(0) \;\; {\it for \; any} \; \eta \in X$.
\item $z^\ast = - z_{xx} \;\; {\it a.e. \; on} \; (0,1)$.
\item $z_x(1) + h(t,1) = 0$ and $z_x(0) + h(t, 0) = 0$.
\end{enumerate}
\end{lemma}
\begin{proof}
Assume $z^\ast = \partial \varphi^t(z)$. First, by \eqref{b2} we have
\begin{align*}
(-z^\ast, \eta)_H =\int^1_0 z \frac{\partial^2 \eta}{\partial x^2} dx \;\; {\rm for} \; \eta \in C^\infty_0(0, 1).
\end{align*}
This implies (B3) and (B1). Also, (B2), (B4)  are direct consequences of  \eqref{b2}.

The sufficiency of (B1) - (B4) is obvious.
\end{proof}

\begin{lemma}\label{h2}
Let $p \in W^{1,2}(0, T; H^2(0,1))$. For each $r>0$, there exist $\alpha_r \in W^{1,2}(0,T)$ and $\beta_r \in W^{1,1}(0, T)$ satisfying: For each $s, t \in [0, T]$ with $s \leq t$ and $z \in D(\varphi^s)$ with $|z|_H \leq r$ there exist $z_1 \in D(\varphi^t)$ such that
\begin{align}
\label{h2-ineq}|z_1 - z|_H &\leq |\alpha_r(t) - \alpha_r(s)|(1+|\varphi^s(z)|^{\frac{1}{2}}),\\
\varphi^t(z_1) - \varphi^s(z) &\leq |\beta_r(t) - \beta_r(s)|(1+|\varphi^s(z)|). \nonumber
\end{align}
\end{lemma}
\begin{proof}
First, by the assumption $p \in W^{1,2}(0, T; H^2(0,1))$ and $h = \frac{\partial p}{\partial x}$, we see that $h \in W^{1,2}(0, T; X) \subset L^\infty(Q(T))$. We have $D(\varphi^t) = X$ for any $t \in [0,T]$. Let $r > 0, s, \; t \in [0, T], s \leq t, z \in D(\varphi^s), |z|_H \leq r$. Here, we put $z_1 = z$ and $\alpha_r(t) = 0$ for $t \in [0, T]$. Easily, we get \eqref{h2-ineq}.

Also, by using Lemma \ref{abs_ineq} and  \eqref{lem33-eq}, we have
\begin{align*}
|z(i)|
&\leq (|z|^2_H + 2|z|_H|z_x|_H)^{\frac{1}{2}}\\
&\leq \sqrt{2}(r^2 + \frac{1}{2}|z_x|_H^2)^{\frac{1}{2}}\\
&\leq \sqrt{2}(r^2 + |\varphi^s(z)| + |h(t,1)z(1)| + |h(t,0)z(0)|)^{\frac{1}{2}} \;\; {\rm for} \; i = 0,1.
\end{align*}
Hence, we get
\begin{align*}
|z(0)| + |z(1)|
&\leq 2\sqrt{2}(r^2 + |\varphi^s(z)| + |h(t,1)z(1)| + |h(t,0)z(0)|)^{\frac{1}{2}}\\
&\leq 2\sqrt{2}(r + |\varphi^s(z)|^{\frac{1}{2}}) + \frac{1}{2}(|z(0)| + |z(1)|) + 8(|h(t,0)| + |h(t,1)|),
\end{align*}
and easily,
\begin{align*}
|z(0)| + |z(1)|
&\leq 4\sqrt{2}(r + |\varphi^s(z)|^{\frac{1}{2}}) + 16(|h(t,0)| + |h(t,1)|)\\
&\leq 4\sqrt{2}r + 2\sqrt{2}(1+|\varphi^s(z)|) + 32|h|_{L^\infty(Q(T))}.
\end{align*}
By putting $C_r = 4\sqrt{2}r + 2\sqrt{2} + 32|h|_{L^\infty(Q(T))} $, we obtain
\begin{align}
\label{lem4eq}|z(0)| + |z(1)| \leq C_r(|\varphi^s(z)| + 1).
\end{align}
Thus, we see that
\begin{align*}
\varphi^t(z_1) - \varphi^s(z) 
&\leq |z(1)||h(t,1) - h(s,1)| + |z(0)||h(t,0) - h(s,0)|\\
&\leq C_r(|\varphi^s(z)| + 1)(|h(t,1) - h(s,1)| + |h(t,0) - h(s,0)|).
\end{align*}
Thanks to Lemma \ref{abs_ineq}, we have
\begin{align*}
|h(t,i) - h(s,i)| 
&\leq |h(t) - h(s)|_{L^\infty(0, 1)}\\
&\leq |h(t) - h(s)|_H + \sqrt{2}|h(t) - h(s)|_H^{\frac{1}{2}} |h(t) - h(s)|_X^{\frac{1}{2}}\\
&\leq\int^t_s |h_\tau(\tau)|_H d\tau + \frac{\sqrt{2}}{2} (|h(t) - h(s)|_H + |h(t) - h(s)|_X)\\
&\leq \left( 1 + \frac{\sqrt{2}}{2} \right) \int^t_s |h_\tau(\tau)|_H d\tau + \frac{\sqrt{2}}{2} \int^t_s |h_\tau(\tau)|_X d\tau\\
&\leq \left( 1 + \sqrt{2} \right) \int^t_s |h_\tau(\tau)|_X d\tau
 \;\; {\rm for} \; i = 0,1.
\end{align*}

Hence, by putting $\beta_r(t) = \int^t_0 2C_r\left( 1 + \sqrt{2} \right) |h_\tau(\tau)|_X d\tau$ for $t \in [0,T]$, we get
\begin{eqnarray*}
\varphi^t(z_1) - \varphi^s(z)
&\leq& 2C_r(|\varphi^s(z)| + 1) \int^t_s \left( 1 + \sqrt{2} \right) |h_\tau(\tau)|_X d\tau \\
&=& (|\varphi^s(z)| + 1)(\beta_r(t) - \beta_r(s)).
\end{eqnarray*}

Since, the assumption of $p$ implies $\beta_r \in W^{1,1}(0, T)$, Lemma \ref{h2} has been proved.
\end{proof}

Next, we show some properties of $B$ without proofs.
\begin{lemma}\label{db_h}
Under the condition (A1) and (A2), it holds that:
\begin{enumerate}[(1)]
\item $D(B) = H$.
\item There exist positive constants $C_0$ and  $C'_0$ such that $$C_0|Bz - Bz_1|_H^2 \leq (Bz - Bz_1, z - z_1)_H, \; C'_0 |z - z_1|_H \leq |Bz - Bz_1|_H \;\; {\rm for} \; z, z_1 \in H.$$
\item Let $J(z) = \int^1_0 \int^{z(x)}_0 B(\xi) d\xi dx \;\; {\it for} \; z \in H$, then $J$ is finite continuous convex function on $H$ and satisfies $j(0) = 0$ and $B = \partial J$. 
\item If $u_0 \in X$, then $\hat{u}_0 = b(v_0) \in H$ and $B(\hat{u}_0) \in D(\varphi^0)$.\\
\end{enumerate}
\end{lemma}

Moreover, we have:
\begin{lemma}
For any $t \in [0, T], r \geq 0$, the set $E = \{ z \in H; |z| \leq r, |\varphi^t(z)| \leq r \}$ is relatively compact in $H$. 
\end{lemma}

\begin{lemma}\label{ap_lem_last}
If $\tilde{u} \in L^2(0,T; X)$ and $p \in W^{1,2}(0, T; H^2(0,1))$, then $f = \frac{\partial}{\partial x} \left\{ \lambda(\tilde{u})\frac{\partial p}{\partial x} \right\} \in L^2(0,T; H)$.
\end{lemma}

\vspace{12pt}
\begin{proof}[Proof of Proposition \ref{AP-SS}]
From Lemmas \ref{ap_lem_first} - \ref{ap_lem_last} and \cite[p.46, Theorem 2.8.1]{Kenmochi}, we have a unique solution $\hat{u} \in W^{1,2}(0,T;H)$ satisfying the following:
\begin{align*}
&\hat{u}'(t) + \partial \varphi^t(B(\hat{u}(t))) = f(t) \;\; {\rm for \; a.e.} \; t\in[0,T], \\
&\hat{u}(0) = \hat{u}_0, \; {\rm the \; function} \; t \mapsto \varphi^t(B(u)) \in L^\infty(0, T).
\end{align*}
Moreover, it holds that 
\begin{align}
\label{sec3_con1}\varphi^t(Bu(t)) \;  {\rm is \; differentiable \;  a.e. \; on} \; [0,T], 
\end{align}
and
\begin{align}
\label{sec3_con2}\varphi^t(Bu(t)) - \varphi^s(Bu(s)) \leq \int^t_s \frac{d}{d \tau} \varphi^\tau(Bu(\tau)) d\tau \;\; {\rm for \; any} \; 0 \leq s \leq t \leq T.
\end{align}
By Lemma \ref{b1-5}, we have 
\begin{align*}
\hat{u}'(t) - \frac{\partial^2 B(\hat{u}(t))}{\partial x^2} 
&= f(t) \;\; {\rm a.e.\; on} \; (0,1) \; {\rm for \; a.e.} \; t \in [0,T],\\
\frac{\partial B(\hat{u})(t, i)}{\partial x} + h(t,i) 
&= 0 \;\; {\rm for} \; i = 0,1 \; {\rm and \; a.e.} \; t \in [0,T].
\end{align*}

Hence, by putting $v = B(\hat{u})$ and Lemma \ref{b1-5} we can prove Proposition \ref{AP-SS}.
\end{proof}

\section{Proof of main theorem}
Throughout this section we suppose (A1), (A2), $p \in W^{1,2}(0,T; H^2(0,1))$ and $u_0 \in X$.
We put $\rho = \psi^{-1}, \hat{\rho}(r) = \int^r_{\mathscr{O}} \rho(s) ds$, where $\mathscr{O}$ is a constant satisfying $\rho(\mathscr{O}) = 0$. Clearly, it holds that
\begin{align}
\label{rho-hat} 
\frac{2C_\psi}{\delta_\psi^2} \hat{\rho} (\psi(u)) 
\geq u^2
\; {\rm and} \;
\hat{\rho}(\psi(u))
\leq \frac{C_\psi}{2} (1 + (\hat{\rho}(\mathscr{O}))^2)(u^2 + 1)
 \;\; {\rm for \; any} \; u \in \mathbb{R}.
\end{align}

\begin{lemma} \label{lem42}
For any $\tilde{u} \in L^2(0,T; X)$ let $u$ be a strong solution of AP($\tilde{u}$) on $Q(T)$. Then, there exists a positive constant $C_1$ independent of $\tilde{u}$ such that:
\begin{align*}
|u(t)|^2_H + \int^{t}_0 \left| u_x(\tau) \right|^2_H d\tau \leq C_1\int^{t}_0  |\tilde{u}_x(\tau)|^2_H d\tau + C_1 \;\; {\it for \; any} \; t \in [0, T].
\end{align*}
\end{lemma}
\begin{proof}
By multiplying both sides of \eqref{AP-eq} with $u$ and integrating it with respect to $x$ over $(0, 1)$, we have
\begin{align*}
\int^1_0 u \frac{\partial}{\partial t} \psi(u) dx
&= \int^1_0 u\frac{\partial}{\partial x}\left( \lambda (u) \frac{\partial u}{\partial x} \right) dx 
+ \int^1_0 u\frac{\partial}{\partial x} \left( \lambda (\tilde{u}) \frac{\partial p}{\partial x} \right) dx\\
& =: I_1 + I_2  \;\; {\rm a.e. \; on} \;[0,T].
\end{align*}
Thanks to the definition of $\hat{\rho}$, the left hand side is rewritten as follows:
\[
\int^1_0 u \frac{\partial}{\partial t} \psi(u) dx
= \frac{d}{dt} \int^1_0 \hat{\rho}(\psi(u)) dx  \;\; {\rm a.e. \; on} \;[0,T].
\]

For $I_1$ integration by parts, Young's inequality and Lemma \ref{abs_ineq} imply
\begin{align*}
I_{1}(t) 
&\leq -\delta_\lambda |u_x(t)|_H^2 + |u(t,1)||p_x(t,1)| + |u(t,0)||p_x(t,0)| \\ 
&\leq -\frac{\delta_\lambda}{2} |u_x(t)|_H^2 + \frac{1}{2} ( |p_x(t,1)|^2 + |p_x(t,0)|^2) + \left( \frac{2}{\delta_\lambda} + 1 \right) |u(t)|_H^2 \;\; {\rm for \; a.e} \; t \in [0,T].
\end{align*}

Also, we have
\begin{align*}
I_{2}
&\leq \int^1_0 \left( \left| u \lambda'(\tilde{u}) \frac{\partial \tilde{u}}{\partial x} \frac{\partial p}{\partial x}\right| + \left| u \lambda(\tilde{u}) \frac{\partial^2 p}{\partial x^2} \right| \right) dx\\
&\leq \frac{C_\lambda}{2} \left( 2|u|^2_H + | p_x |_{L^\infty(Q(T))}^2 |\tilde{u}_x|^2_H + |p_{xx}|^2_H \right) \;\; {\rm a.e. \; on} \;[0,T].
\end{align*}
Thus, by \eqref{rho-hat} and putting $v = \psi(u)$ we get
\begin{align*}
\frac{d}{dt} \int^1_0 \hat{\rho}(v) dx + \frac{\delta_\lambda}{2} |u_x|^2_H
&\leq \frac{1}{2}( |p_x(\cdot,1)|^2 + |p_x(\cdot,0)|^2 + C_\lambda|p_{xx}|^2_H)\\
&\quad + \frac{2C_\psi}{\delta_\psi^2}\left( C_\lambda + \frac{2}{\delta_\lambda} + 1 \right) \int^1_0 \hat{\rho}(\psi(u)) dx
+ \frac{C_\lambda}{2} | p_x |_{L^\infty(Q(T))}^2 |\tilde{u}_x|^2_H \\
&= F_1(t) + \hat{C}_1 \int^1_0 \hat{\rho}(\psi(u)) dx + \frac{C_\lambda}{2}|p_x|^2_{L^\infty(Q(T))}|\tilde{u}_x|^2_H \;\; {\rm a.e. \; on} \; [0,T],
\end{align*}
where $F_1 = \frac{1}{2}( |p_x(\cdot,1)|^2 + |p_x(\cdot,0)|^2 + C_\lambda|p_{xx}|^2_H), \hat{C}_1 = \frac{2C_\psi}{\delta_\psi^2}\left( C_\lambda + \frac{2}{\delta_\lambda} + 1 \right)$. 

By using Gronwall's inequality and \eqref{rho-hat}, we  obtain  
\begin{align*}
&\int^1_0 \hat{\rho}(v(t')) dx + \frac{\delta_\lambda}{2} \int^{t'}_0 |u_x(\tau)|^2_H d\tau\\
&\leq e^{\hat{C}_{1}T} \left( \int^{t'}_0 \left( F_1(t) + \frac{C_\lambda}{2}|p_x|^2_{L^\infty(Q(T))} | \tilde{u}_x(t) |_H^2 \right) dt + \int^1_0 \hat{\rho}(v(0)) dx \right) \;\; {\rm for \; any} \; t \in [0, T], 
\end{align*}
\vspace{-6mm}
\begin{align*}
\hat{\rho}(v(0))
\leq \hat{C}_2(|u_0^2 + 1|) \;\; {\rm on } \; (0,1),
\end{align*}
where $\hat{C}_2$ is a positive constant. 

Accordingly, we have
\begin{align*}
&|u(t')|_H^2 +  \int^{t'}_0 |u_x(\tau)|^2_H d\tau\\
&\leq e^{C_{1}T} \frac{C_\lambda}{2} |p_x|^2_{L^\infty(Q(T))} \hat{C}_3 \int^{t'}_0 | \tilde{u}_x(t) |_H^2 dt 
+ e^{\hat{C}_{1}T} \hat{C}_3 \left( \int^{t'}_0 F_1(t) dt + |u_0|_H^2 + 1 \right)\\
&= C_1\int^{t'}_0 | \tilde{u}_x(t) |_H^2 dt  + C_1 \;\; {\rm for \; any}\; t' \in [0,T],
\end{align*}
where $\hat{C}_3 = \max \left\{ \frac{2C_\psi}{\delta_\psi^2}, \frac{2}{\delta_\lambda}, 1, \hat{C}_2 \right\}, C_1 = \max \bigg\{ e^{C_{1}T} \frac{C_\lambda}{2} |p_x|^2_{L^\infty(Q(T))} \hat{C}_3,\\ e^{\hat{C}_{1}T}\hat{C}_3 \left( \int^T_0 F_1(t) dt + |u_0|_H^2 + 1 \right) \bigg\}$. Thus we prove this lemma.
\end{proof}

\begin{lemma}\label{lem43}
Under the same assumption as in Lemma \ref{lem42}, there exists positive constant $C_2$ independent of $\tilde{u}$ such that
\begin{align*}
|u_x(t)|^2_H + \int^{t}_0 \left| u_\tau(\tau) \right|^2_H d\tau \leq C_2 \int^{t}_0  |\tilde{u}_x(\tau)|^2_H d\tau + C_2 \;\; {\it for \; any} \; t \in [0, T].
\end{align*}
\end{lemma}
\begin{proof}
First, by using $\hat{\lambda}$ we see that
\begin{align}
\label{hat-lambda}\frac{\partial \psi(u)}{\partial t} 
&= \frac{\partial^2}{\partial x^2} \hat{\lambda}(u) + \frac{\partial}{\partial x} \left( \lambda(\tilde{u})\frac{\partial p}{\partial x} \right) \;\; {\rm in} \; Q(T).
\end{align}
Multiply it by $\frac{\partial \hat{\lambda}(u)}{\partial t}$ and integrate it with respect to $x$ over $[0, 1]$, then we have
\begin{align}
\int^1_0 \frac{\partial \psi(u)}{\partial t} \frac{\partial \hat{\lambda}(u)}{\partial t} dx 
&= \int^1_0 \frac{\partial}{\partial x} \left\{ \frac{\partial \hat{\lambda}(u)}{\partial x} \right\} \frac{\partial \hat{\lambda}(u)}{\partial t} dx
 + \int^1_0 \frac{\partial}{\partial x} \left\{ \lambda(\tilde{u})\frac{\partial p}{\partial x} \right\} \frac{\partial \hat{\lambda}(u)}{\partial t} dx \nonumber \\
&\label{lem43eq}=: I_3 + I_4 \;\; {\rm a.e. \; on} \;[0,T].
\end{align}
It is easy to see that
\[
\int^1_0 \frac{\partial \psi(u)}{\partial t} \frac{\partial \hat{\lambda}(u)}{\partial t} dx \geq \delta_\psi \delta_\lambda |u_t|^2_H \;\; {\rm a.e. \; on} \;[0,T].
\]
By elementary calculation, (A2) and \eqref{sec3_con1}, we obtain
\begin{align*}
I_{3} 
&\leq \frac{d}{dt} \left\{ -\frac{1}{2} \int^1_0 \left( \frac{\partial \hat{\lambda}(u)}{\partial x} \right)^2 dx
 - \frac{\partial p(\cdot,1)}{\partial x} \hat{\lambda}(u(\cdot,1)) + \frac{\partial p(\cdot,0)}{\partial x} \hat{\lambda}(u(\cdot,0)) \right\} \\
&\quad + \frac{1}{2\delta_\lambda^2}(|\hat{\lambda}(u(t,1))|^2 + |\hat{\lambda}(u(t,0))|^2) + \frac{C_\lambda^2}{2} \left( \left| \frac{\partial}{\partial t} \left( \frac{\partial p(\cdot,1)}{\partial x} \right) \right|^2 + \left| \frac{\partial}{\partial t} \left( \frac{\partial p(\cdot,0)}{\partial x} \right) \right|^2 \right) \\
&\hspace{130mm} {\rm a.e. \; on} \;[0,T]
\end{align*}
By applying Young's inequality to $I_4$, it holds that
\begin{align*}
I_4
&\leq \frac{\delta_\psi \delta_\lambda}{2} |u_t|_H^2  
+ \frac{1}{2 \delta_\psi \delta_\lambda}\int^1_0 \left( \left( \lambda'(\tilde{u}) \cdot \frac{\partial \tilde{u}}{\partial x} \cdot \frac{\partial p}{\partial x} + \lambda(\tilde{u})\frac{\partial^2 p}{\partial x^2} \right) \cdot \lambda(u) \right)^2 dx\\
&\leq \frac{\delta_\psi \delta_\lambda}{2} |u_t|_H^2  
+ \frac{2C_\lambda^4}{2 \delta_\psi \delta_\lambda}\int^1_0 \left\{ \left( \frac{\partial \tilde{u}}{\partial x} \cdot \frac{\partial p}{\partial x}  \right)^2 +\left( \frac{\partial^2 p}{\partial x^2} \right)^2 \right\} dx\\
&\leq \frac{\delta_\psi \delta_\lambda}{2} |u_t|_H^2  
+ \frac{2C_\lambda^4}{2 \delta_\psi \delta_\lambda}|p_x|_{L^\infty(Q(T))}^2|\tilde{u}_x|^2_H + \frac{2C_\lambda^4}{2 \delta_\psi \delta_\lambda} |p_{xx}|_H^2 \;\; {\rm a.e. \; on} \; [0,T].
\end{align*}
Obviously, $\hat{\lambda}(u) \in X$ a.e. on $[0,T]$. Accordingly, we can apply Lemma \ref{abs_ineq} to $\hat{\lambda}(u)$ and show that
\[
|\hat{\lambda}(u(t, i))|^2 \leq 2|\hat{\lambda}(u)|_H^2 + |(\hat{\lambda}(u))_x|_H^2 \;\; {\rm for \; a.e.} \; t \in [ 0, T ].
\]
Therefore, from \eqref{lem43eq} and these inequalities, it follows that
\begin{align}
\label{lem42eq1}&\frac{\delta_\psi \delta_\lambda}{2} |u_t|_H^2 
+ \frac{d}{dt} \left\{ \frac{1}{2} \int^1_0 \left( \frac{\partial \hat{\lambda}(u)}{\partial x} \right)^2 dx \nonumber
+ \frac{\partial p(\cdot,1)}{\partial x} \cdot \hat{\lambda}(u(\cdot,1)) 
- \frac{\partial p(\cdot,0)}{\partial x} \cdot \hat{\lambda}(u(\cdot,0))\right\}\\
&\leq \frac{1}{\delta_\lambda^2}(2|\hat{\lambda}(u)|_H^2  + |(\hat{\lambda}(u))_x|_H^2) 
+ \frac{2C_\lambda^4}{2 \delta_\psi \delta_\lambda} |p_{xx}|_H^2
+ \frac{2C_\lambda^4}{2 \delta_\psi \delta_\lambda}|p_x|_{L^\infty(Q(T))}^2|\tilde{u}_x|^2_H \nonumber \\
&\quad+ \frac{C_\lambda^2}{2} \left( \left| \frac{\partial}{\partial t} \left( \frac{\partial p(\cdot,1)}{\partial x} \right) \right|^2 + \left| \frac{\partial}{\partial t} \left( \frac{\partial p(\cdot,0)}{\partial x} \right) \right|^2 \right) \\
& =: \frac{1}{\delta_\lambda^2}(2|\hat{\lambda}(u)|_H^2  + |(\hat{\lambda}(u))_x|_H^2) 
+ \frac{2C_\lambda^4}{2 \delta_\psi \delta_\lambda}|p_x|_{L^\infty(Q(T))}|\tilde{u}_x|^2_H
+ F_2 \;\; {\rm a.e. \; on} \;[0,T]. \nonumber
\end{align}
By the assumption for $p$ and Lemma \ref{abs_ineq}, we have
\begin{align*}
&\left| \frac{\partial p(\cdot, i)}{\partial x} \hat{\lambda}(u(t, i)) \right|
\leq |p_x|_{L^\infty(Q(T))} \left( \sqrt{2}|\hat{\lambda}(u(t))|_H + \left| \frac{\partial}{\partial x} \hat{\lambda}(u(t)) \right|_H \right)\\
&\hspace{100mm}{\rm for} \; i = 0, 1 \; {\rm and \; a.e.} \; t \in [0, T].
\end{align*} 
Due to this inequality and the definition of $\varphi^t$, we obtain 
\begin{align*}
\varphi^t(\hat{\lambda}(u(t)))
&= \frac{1}{2} \int^1_0 \left| \frac{\partial}{\partial x} (\hat{\lambda} (u(t))) \right|^2 dx + \frac{\partial p(t, 1)}{\partial x}\hat{\lambda} (u(t, 1))
- \frac{\partial p(t, 0)}{\partial x}\hat{\lambda} (u(t, 0))\\
&\geq \frac{1}{2} \int^1_0 \left| \frac{\partial}{\partial x} (\hat{\lambda} (u(t))) \right|^2 dx
- 2|p_x|_{L^\infty(Q(T)} \left( \sqrt{2} |\hat{\lambda}(u(t))|_H +  \left| \frac{\partial}{\partial x} \hat{\lambda}(u(t)) \right|_H \right)\\
&\geq \frac{1}{4} \int^1_0 \left| \frac{\partial}{\partial x} (\hat{\lambda} (u(t))) \right|^2 dx
- 2\sqrt{2} |p_x|_{L^\infty(Q(T)} |\hat{\lambda}(u(t))|_H
- 4 |p_x|_{L^\infty(Q(T)}^2\\
&\geq - 2\sqrt{2} |p_x|_{L^\infty(Q(T)} |\hat{\lambda}(u(t))|_H
- 4 |p_x|_{L^\infty(Q(T)}^2 \;\; {\rm for \; a.e.} \; t \in [0,T]. 
\end{align*}
Here, thanks to Lemma \ref{lem42} and \eqref{func_ineq}, we have $|\hat{\lambda}(u)|_{L^\infty(0,T; H)} < + \infty$. Accordingly, from \eqref{lem42eq1} it follows that
\begin{align*}
&\frac{\delta_\psi \delta_\lambda}{2} |u_t|_H^2
+ \frac{d}{dt} (\varphi^t(\hat{\lambda}(u))) + 2\sqrt{2} |p_x|_{L^\infty(Q(T)} |\hat{\lambda}(u(t))|_H
+ 4 |p_x|_{L^\infty(Q(T)}^2\\
&\leq \hat{C}_4 (|\hat{\lambda}(u)|_H^2 + \varphi^t(\hat{\lambda}(u)) + 2\sqrt{2} |p_x|_{L^\infty(Q(T)} |\hat{\lambda}(u(t))|_H
+ 4 |p_x|_{L^\infty(Q(T)}^2\\
&\quad + \hat{C}_4 (|p_x|_{L^\infty(Q(T))}^2 + |p_x|_{L^\infty(Q(T))}|\tilde{u}_x|_H + F_2) \;\; {\rm a.e. \; on} \; [0, T],
\end{align*}
where $\hat{C}_4$ is a positive constant independent of $\tilde{u}$. By applying Gronwall's inequality, we infer that 
\begin{align*}
&\frac{\delta_\psi \delta_\lambda}{2} \int^t_0 |u_\tau|_H^2 d\tau 
+ \varphi^t(\hat{\lambda}(u)) + 2|p_x|_{L^\infty(Q(T))} |\hat{\lambda}(u)|_{L^\infty(0,T; H)}
+ 8 |p_x|_{L^\infty(Q(T))}^2\\
&\leq e^{\hat{C}_4 t} (\varphi^0(\hat{\lambda}(u)) + 2|p_x|_{L^\infty(Q(T))} |\hat{\lambda}(u)|_{L^\infty(0,T; H)} + 8 |p_x|_{L^\infty(Q(T))}^2)\\
&\quad + \hat{C}_4 e^{\hat{C}_4 t} \int^t_0 (|\hat{\lambda}(u)|_H + 2|p_x|_{L^\infty(Q(T))} |\hat{\lambda}(u)|_{L^\infty(0,T; H)} + 8 |p_x|_{L^\infty(Q(T))}^2) d\tau \\
&\quad + \hat{C}_4 e^{\hat{C}_4 t} \int^t_0 (|p_x|_{L^\infty(Q(T))}^2 + |p_x|_{L^\infty(Q(T))}|\tilde{u}_x|_H + F_2) d\tau \;\; {\rm for} \; t \in [0, T].
\end{align*}
Due to \eqref{lem33-eq} and Lemma \ref{lem42}, we can easily prove this lemma.
\end{proof}

Let $\Gamma_T : L^2(0,T; X) \to L^2(0,T; X)$ be the solution operator for AP($\tilde{u}$), namely, $\Gamma_T(\tilde{u}) = u$ for $\tilde{u} \in L^2(0, T; X)$. Also, for any $T' \in (0 ,T], M > 0$, we put
\[
K(T', M) = \{ z \in L^2(0, T; X) ; |z_x(t)|_H^2 \leq M \; \mathrm{for \; a.e.} \; t \in (0, T') \}.
\]
The following lemma concerned with $\Gamma_{T'}$ and $K(T', M)$ holds.
\begin{lemma}\label{KTM}
Under the same assumption as in Lemma \ref{lem42}, there exists $T' \in (0, T]$ and $M > 0$ such that $\Gamma_{T'}: K(T', M) \to K(T', M)$.
\end{lemma}
\begin{proof}
Let $M \geq 2C_2$ and $0 < T' \leq \min{\left\{ \frac{1}{2C_2} , T \right\}}$, where $C_2$ is the constant given in Lemma \ref{lem43}. For $\tilde{u} \in L^2(0,T; X)$, let $u = \Gamma_T(\tilde{u})$. By Lemma \ref{lem43}, we have
\begin{align*}
\int^1_0 |u_x(t)|^2 dx 
&\leq C_2 + C_2 T' M\\
&\leq \frac{M}{2} + \frac{M}{2}\\
&= M \;\; {\rm for \; any} \; t \in [0, T'].
\end{align*}
Thus we get $u \in K(T', M)$. This is a conclusion of this lemma.
\end{proof}

Moreover, by Lemma \ref{lem43} we get some uniform estimate for $u = \Gamma_{T'}(\tilde{u})$.
\begin{lemma}\label{2esti}
Assume the same condition as in Lemma \ref{lem42} and let $T'$ and $M$ be constants defined by Lemma \ref{KTM}. Then there exist positive constants $C_3$ and $C_4$ such that for any $\tilde{u} \in K(T', M), u = \Gamma_{T'} (\tilde{u})$ satisfies
\begin{align}
\label{cor41eq1}\int^{T'}_0 |u_x|^4_{L^4(0,1)} dt \leq C_3, \; \int^{T'}_0 |u_{xx}|^2_{H} dt \leq C_4.
\end{align}
\end{lemma}
\begin{proof}
For any $\tilde{u} \in K(T', M)$, let $u = \Gamma_T(\tilde{u})$. In this proof positive constants independent of $\tilde{u}$ are written as $C$. By \eqref{hat-lambda} and Lemma \ref{lem43}, we have
\[
\int^{t'}_0 |(\hat{\lambda}(u))_{xx}|_H^2 dt \leq C \;\; {\rm for} \; t \in [0, T'].
\]
Also, by Lemmas \ref{abs_ineq}, \ref{lem42} and \ref{lem43}, we get
\begin{align*}
|\hat{\lambda}(u(t, x))| \leq C \;\; {\rm for} \; (t, x) \in Q(T'),
\end{align*}
and
\begin{align*}
\left| \left( \frac{\partial \hat{\lambda}(u(t, i))}{\partial x} \right)^3 \hat{\lambda}(u(t, i))  \right|
\leq \left| {\left(-\frac{\partial p(t, i)}{\partial x} \right)^3} \hat{\lambda}(u(t, i)) \right| \leq C \;\; {\rm for} \; i = 0, 1 \;{\rm and} \; t \in [0, T']. 
\end{align*}
From these inequalities, we obtain
\begin{align*}
\int^1_0 \left( \frac{\partial \hat{\lambda}(u)}{\partial x} \right)^4 dx
&= -\int^1_0 \frac{\partial}{\partial x} \left( \frac{\partial \hat{\lambda}(u)}{\partial x} \right)^3 \cdot \hat{\lambda}(u) dx + \left[ \left( \frac{\partial \hat{\lambda}(u)}{\partial x} \right)^3 \cdot \hat{\lambda}(u) \right]^1_0\\
&\leq \frac{1}{2}\int^1_0  \left( \frac{\partial \hat{\lambda}(u)}{\partial x} \right)^4 dx
+ C \int^1_0  \left| \frac{\partial^2 \hat{\lambda}(u)}{\partial x^2} \right|^2  dx + C
 \;\; {\rm a.e. \; on} \; [0, T]
\end{align*}
Therefore, because of $\left| \frac{\partial \hat{\lambda}(u)}{\partial x} \right| = \left| \lambda(u) \frac{\partial u}{\partial x} \right| \geq \delta_\lambda \left| \frac{\partial u}{\partial x} \right|$, we can get the first inequality of \eqref{cor41eq1}. 

Next, from \eqref{AP-eq} we can easily set
\begin{align*}
\lambda(u) u_{xx} = \frac{\partial}{\partial t} \psi(u) - \frac{\partial \lambda(u)}{\partial x} \frac{\partial u}{\partial x} - \frac{\partial \lambda(\tilde{u})}{\partial x} \frac{\partial p}{\partial x} - \lambda(\tilde{u}) \frac{\partial^2 p}{\partial x^2} \;\; {\rm in} \; Q(T).
\end{align*}
Accordingly, by (A1), (A2) and the first inequality of \eqref{cor41eq1} we have
\begin{align*}
\int^t_0 |u_{xx}|_H^2 d\tau
&\leq C \left(\int^t_0 |u_\tau|_H^2 d\tau
+\int^t_0 \int^1_0 \left| \lambda'(u) \left( u_x \right)^2 \right|^2 dx d\tau
+ \int^t_0 |\tilde{u}_x|_H^2 d\tau
+ \int^t_0 |p_{xx}|_H^2 d\tau \right)\\
&\leq C_4 \;\; {\rm for \; any} \; t \in [0, T],
\end{align*}
where $C_4$ is a positive constant independent of $\tilde{u}$. Thus, we have the second inequality of \eqref{cor41eq1}.
\end{proof}

The next two lemmas are necessary to apply the Banach fixed point theorem to $\Gamma_{T'}$.
\begin{lemma}\label{lem45}
Under the same assumptions as in Lemma \ref{2esti}, there exists a positive constant $C_5$ satisfying the following estimate:
\begin{align}
\label{lem45-eq}|(u_1 - u_2)(t)|^2_H\leq C_5 \int^{t}_{0} |\tilde{u}_1 - \tilde{u}_2|^2_X d\tau \;\; {\it for \; any } \; t \in [0, T'] \; {\it and} \; \tilde{u}_i \in K(M, T'),
\end{align}
where $u_i = \Gamma_{T'}(\tilde{u}_i)$ for $i = 1, 2$.
\end{lemma}
\begin{proof}
From \eqref{AP-eq}, we have
\begin{align}
\label{sub-est}\frac{\partial \psi(u_1) }{\partial t} - \frac{\partial \psi(u_2)}{\partial t} = \frac{\partial}{\partial x}\left( \lambda (u_1) u_{1x}- \lambda (u_2)u_{2x} \right) + \frac{\partial}{\partial x}\left( \lambda (\tilde{u}_1) p_x - \lambda (\tilde{u}_2) p_x \right) \;\; {\rm in} \; Q(T).
\end{align}
Here, we put $u = u_1 - u_2, \tilde{u} = \tilde{u}_1 - \tilde{u}_2$. Multiply \eqref{sub-est} by $u$ and integrate it, then we get
\begin{align*}
&\int^1_0 u \psi'(u_1) u_t dx
+ \int^1_0 u u_{2t} (\psi'(u_1) - \psi'(u_2)) dx\\
&= \int^1_0 u \frac{\partial}{\partial x}\left( \lambda (u_1) u_{1x}- \lambda (u_2) u_{2x} \right)dx 
+ \int^1_0 u\frac{\partial}{\partial x}\left( \lambda (\tilde{u}_1) p_x - \lambda (\tilde{u}_2) p_x \right) dx \;\; {\rm a.e. \; on} \; [0, T].
\end{align*}
For the first term on the left side and the first term on the right side, (A1), (A2) and elementary calculation guarantee
\begin{align*}
\int^1_0 u \psi'(u_1) u_t dx
&= \frac{1}{2}\int^1_0 \psi'(u_1) (u^2)_t dx\\
&= \frac{1}{2}\int^1_0 \left( \frac{\partial}{\partial t}\left( \psi'(u_1) u^2 \right) - \frac{\partial \psi'(u_1)}{\partial t}u^2 \right)dx\\
&= \frac{1}{2}\frac{d}{dt} \int^1_0 \psi'(u_1) u^2 dx
- \frac{1}{2} \int^1_0 \psi''(u_1) u_{1t} u^2 dx, 
\end{align*}
\begin{align*}
&\int^1_0 u \frac{\partial}{\partial x}\left( \lambda (u_1) u_{1x}- \lambda (u_2) u_{2x} \right)dx\\
&= -\int^1_0 u_x \left( \lambda (u_1) u_{1x} - \lambda (u_2) u_{2x}\right) dx\\
&= -\int^1_0 (u_x)^2 \lambda(u_1) dx
-\int^1_0 u_x u_{2x} \left( \lambda (u_1) - \lambda (u_2)\right) dx\\
&\leq -\delta_\lambda\int^1_0  (u_x)^2 dx
-\int^1_0 u_x u_{2x} \left( \lambda (u_1) - \lambda (u_2)\right) dx \;\; {\rm a.e. \; on} \; [0, T].\\ 
\end{align*}
From these inequalities, we have
\begin{align} \label{las2lem-eq1}
&\frac{1}{2}\frac{d}{dt} \int^1_0 \psi'(u_1) u^2 dx
+ \delta_\lambda \int^1_0 (u_x)^2 dx \nonumber \\
&\leq \frac{1}{2} \int^1_0 \psi''(u_1) u_{1t} u^2 dx
-\int^1_0 u u_{2t} (\psi'(u_1) - \psi'(u_2)) dx \nonumber \\
&-\int^1_0 u_x u_{2x} \left( \lambda (u_1) - \lambda (u_2)\right) dx
+ \int^1_0 u \frac{\partial}{\partial x}\left( \lambda (\tilde{u}_1) p_x - \lambda (\tilde{u}_2) p_x \right) dx \nonumber \\
&=: \sum^8_{i = 5} I_i \;\; {\rm a.e. \; on} \; [0, T].
\end{align}
By H\"older's inequality, Lemmas \ref{abs_ineq} and \ref{lem43}, we get
\begin{align*}
I_{5}
&\leq \frac{1}{2} \int^1_0 |\psi''(u_1)| | u_{1t}| |u|^2 dx\\
&\leq \frac{C_\psi}{2} \left\{ \int^1_0 \left| u_{1t} \right|^2 dx \right\}^{\frac{1}{2}}\left\{ \int^1_0 |u|^4 dx \right\}^{\frac{1}{2}}\\
&\leq \frac{C_\psi}{2}|u_{1t}|_H(|u|_H^2 + 2|u|_H|u_x|_H)\\
&\leq \left( \frac{C_\psi}{2} + \frac{C_\psi^2}{\delta_\lambda}|u_{1t}|_H\right)|u_{1t}|_H|u|_H^2 + \frac{\delta_\lambda}{4}|u_x|_H^2,\\
\stepcounter{equation}\tag{\theequation}
I_{6}
&\leq \int^1_0 |u| \left| u_{2t} \right|| \psi'(u_1) -\psi'(u_2)| dx\\
&\leq C_\psi \left\{ \int^1_0 |u|^4 dx \right\}^{\frac{1}{2}} \left\{ \int^1_0 \left| u_{2t} \right|^2 dx \right\}^{\frac{1}{2}}\\
&\leq C_\psi |u_{2t}|_H(|u|_H^2 + 2|u|_H|u_x|_H)\\
&\leq \left( C_\psi + \frac{4C_\psi^2}{\delta_\lambda}|u_{2t}|_H \right)|u_{2t}|_H|u|_H^2 + \frac{\delta_\lambda}{4}|u_x|_H^2,\\
\stepcounter{equation}\tag{\theequation}
I_{7}
&\leq \int^1_0 |u_x| |u_{2x}| |\lambda (u_1) - \lambda (u_2)| dx\\
&\leq \left\{ \int^1_0 |u_x|^2 dx\right\}^{\frac{1}{2}} \left\{ \int^1_0 |u_{2x}|^2 \left| \lambda (u_1) - \lambda (u_2)\right|^2 dx\right\}^{\frac{1}{2}}\\
&\leq \frac{\delta_\lambda}{8}\int^1_0 |u_x|^2 dx
+ \frac{2C_\lambda^2}{\delta_\lambda}\int^1_0 \left|\frac{\partial u_2}{\partial x}\right|^2 | u_1 - u_2 |^2 dx\\
&\leq \frac{\delta_\lambda}{8}|u_x|_H^2
+ \frac{2C_\lambda^2}{\delta_\lambda}(|u|_H^2 + 2|u|_H|u_x|_H)|u_{2x}|_H^2\\
&\leq \frac{\delta_\lambda}{8}|u_x|_H^2
+ \frac{2C_\lambda^2}{\delta_\lambda}( C_{2}+C_{2}MT')(|u|_H^2 + 2|u|_H|u_x|_H)\\
&\leq \frac{\delta_\lambda}{4}|u_x|_H^2
+ \left(
\frac{2C_\lambda^2}{\delta_\lambda}( C_{2}+C_{2}MT')
+ \frac{2}{\delta_\lambda}\left( \frac{2C_\lambda^2}{\delta_\lambda} \cdot 2 ( C_{2}+C_{2}MT') \right)^2
\right) |u|_H^2 \\
&= \frac{\delta_\lambda}{4}|u_x|_H^2
+ \hat{C}_5|u|_H^2 \;\; {\rm a.e. \; on} \; [0, T],
\stepcounter{equation}\tag{\theequation}
\end{align*}
where $\hat{C}_5$ is a suitable positive constant.
Also, applying integration by parts to $I_8$ yields
\begin{align*}
I_{8}
&= -\int^1_0 u_x \left( \lambda (\tilde{u}_1) p_x - \lambda (\tilde{u}_2) p_x \right) dx
+ u(\cdot, 1) p_x(\cdot, 1) \left( \lambda (\tilde{u}_1(\cdot, 1))  - \lambda (\tilde{u}_2(\cdot, 1)) \right)\\
&\quad - u(\cdot, 0) p_x(\cdot, 0) \left( \lambda (\tilde{u}_1(\cdot, 0))  - \lambda (\tilde{u}_2(\cdot, 0)) \right)\\
&\leq C_\lambda \left| p_x \right|_{L^\infty(Q(T))} \int^1_0 |u_x||\tilde{u}| dx
+ u(\cdot, 1) p_x(\cdot, 1) \left( \lambda (\tilde{u}_1(\cdot, 1))  - \lambda (\tilde{u}_2(\cdot, 1)) \right)\\
&\quad - u(\cdot, 0) p_x(\cdot, 0) \left( \lambda (\tilde{u}_1(\cdot, 0))  - \lambda (\tilde{u}_2(\cdot, 0)) \right)\\
&\leq C_\lambda \left| p_x \right|_{L^\infty(Q(T))}
(|u_x|_H|\tilde{u}|_H + |u(\cdot, 1)||\tilde{u}(\cdot, 1)| 
+ |u(\cdot, 0)||\tilde{u}(\cdot, 0)|) \; {\rm a.e. \;on} \; [0, T].
\end{align*}
For $i = 0, 1$, by using \eqref{abs_ineq}, we get
\begin{align*}
&|u(t, i)||\tilde{u}(t, i)|\\
&\leq (|u(t)|_H^2 + 2|u(t)|_H|u_x(t)|_H)^{\frac{1}{2}}(|\tilde{u}(t)|_H^2 + 2|\tilde{u}(t)|_H|\tilde{u}_x(t)|_H)^{\frac{1}{2}}\\
&\leq (|u(t)|_H + \sqrt{2}|u(t)|_H^{\frac{1}{2}}|u_x(t)|_H^{\frac{1}{2}})(|\tilde{u}(t)|_H + \sqrt{2}|\tilde{u}(t)|_H^{\frac{1}{2}}|\tilde{u}_x(t)|_H^{\frac{1}{2}})\\
&=|u(t)|_H|\tilde{u}(t)|_H + 2|u(t)|_H^{\frac{1}{2}}|u_x(t)|_H^{\frac{1}{2}}|\tilde{u}(t)|_H^{\frac{1}{2}}|\tilde{u}_x(t)|_H^{\frac{1}{2}}
+ \sqrt{2}|u(t)|_H|\tilde{u}(t)|_H^{\frac{1}{2}}|\tilde{u}_x(t)|_H^{\frac{1}{2}}\\
&\quad + \sqrt{2}|\tilde{u}(t)|_H|u(t)|_H^{\frac{1}{2}}|u_x(t)|_H^{\frac{1}{2}} \; {\rm for \; a.e.} \; t \in [0,T'].
\end{align*}
By Young's inequality, it follows that
\begin{align}
I_{8}(t)
&\leq 4 C_\lambda |p_x|_{L^\infty(Q(T))} (|u_x(t)|_H |\tilde{u}(t)|_H + |u(t)|_H |\tilde{u}(t)|_H + |u(t)|_H |\tilde{u}_x(t)|_H) \nonumber\\
&\leq \frac{3\delta_\lambda}{32}|u_x(t)|_H^2
+ \left(\frac{8 \cdot 4^2}{3 \delta_\lambda}C_\lambda \left| p_x \right|_{L^\infty(Q(T))}
+ \frac{4}{2} \right) C_\lambda \left| p_x \right|_{L^\infty(Q(T))} |\tilde{u}(t)|_H^2 \nonumber\\
&\quad + 4C_\lambda \left| p_x \right|_{L^\infty(Q(T))} |u(t)|_H^2
+ 2C_\lambda \left| p_x \right|_{L^\infty(Q(T))} |\tilde{u}_x(t)|_H^2 \nonumber\\
&\label{las2lem-eq2}\leq \frac{3 \delta_\lambda}{32}|u_x(t)|_H^2
+\hat{C}_6(|u(t)|_H^2 
+ |\tilde{u}(t)|_H^2 + |\tilde{u}_x(t)|_H^2) \;\; {\rm for} \; t \in [0, T'],
\end{align}
where
\[
\hat{C}_6 = \left( \frac{8 \cdot 4^2}{3 \delta_\lambda}C_\lambda \left| p_x \right|_{L^\infty(Q(T))} + 4 \right) C_\lambda \left| p_x \right|_{L^\infty(Q(T))}.
\]
From \eqref{las2lem-eq1} - \eqref{las2lem-eq2}, we have
\begin{align*}
&\frac{1}{2}\frac{d}{dt} \int^1_0 \psi'(u_1) (u_1 - u_2)^2 dx
+ \frac{5\delta_\lambda}{32}\int^1_0  \left( \frac{\partial u_1}{\partial x} - \frac{\partial u_2}{\partial x}\right)^2 dx\\
&\leq \left(
\left( \frac{C_\psi}{2} + \frac{C_\psi^2}{\delta_\lambda}|u_{1t}|_H\right)|u_{1t}|_H
+\left( C_\psi + \frac{4C_\psi^2}{\delta_\lambda}|u_{2t}|_H \right)|u_{2t}|_H
+ \hat{C}_5 + \hat{C}_6
\right)|u|_H^2\\
&\quad +\hat{C}_6|\tilde{u}|_H^2 
+\hat{C}_6|\tilde{u}_x|_H^2\\
&\leq \left(
 \frac{C_\psi^2}{8} + \frac{1}{2}|u_{1t}|_H^2 + \frac{C_\psi^2}{\delta_\lambda}|u_{1t}|_H^2
+ \frac{C_\psi^2}{2} + \frac{1}{2}|u_{2t}|_H^2 + \frac{4C_\psi^2}{\delta_\lambda}|u_{2t}|_H^2
+ \hat{C}_5 + \hat{C}_6
\right)|u|_H^2\\
&\quad +\hat{C}_6|\tilde{u}|_H^2 
+\hat{C}_6|\tilde{u}_x|_H^2\\
&\leq \left(
\left(\frac{1}{2} + \frac{4C_\psi^2}{\delta_\lambda} \right)
\left( |u_{1t}|_H^2 + |u_{2t}|_H^2 \right) + \frac{C_\psi^2}{8} + \frac{C_\psi^2}{2} + \hat{C}_5 + \hat{C}_6
\right)|u|_H^2\\
&\quad +\hat{C}_6|\tilde{u}|_H^2 
+\hat{C}_6|\tilde{u}_x|_H^2 \;\;{\rm a.e. \; on} \; [0, T'].
\end{align*}
Hence, there exists a positive constant 
$\hat{C}_7$ satisfying
\begin{align*}
&\frac{1}{2}\frac{d}{dt} \int^1_0 \left( \psi'(u_1) (u)^2 \right) dx
+ \frac{5\delta_\lambda}{32} |u_x|^2_H\\
&\hspace{30mm}\leq\hat{C}_7(|u_{1t}|_H^2 + |u_{2t}|_H^2 +1 ) |u|_H^2
+ \hat{C}_7|\tilde{u}_1 - \tilde{u}_2|_X^2 
\;{\rm a.e. \; on} \; [0, T'].
\end{align*}
By Gronwall's inequality, we have
\begin{align*}
\frac{\delta_\psi}{2}  \int^1_0 (u(t))^2  dx
&\leq \exp{ \left\{ \frac{2 \hat{C}_7}{\delta_\psi} \left(\int^t_0 |u_{1\tau}|_H^2 d\tau + \int^t_0 |u_{2\tau}|_H^2 d\tau + t \right) \right\} } \hat{C}_7 |\tilde{u}|_X^2\\
&\hspace{87.5mm}{\rm for \; any} \; t \in [0, T']. \nonumber
\end{align*}
From Lemma \ref{lem43} and the definition of $K(T', M)$, \eqref{lem45-eq} is shown.
\end{proof}

\begin{lemma}\label{lem46}
Under the same assumption as in Lemma \ref{2esti}, there exists a positive constant $C_6$ satisfying
\begin{align}\label{diff-eq2}
|(u_1 - u_2)_x(t)|^2_H \leq C_6\int^{t}_0 |\tilde{u}_1 - \tilde{u}_2|^2_X d\tau \;\; {\it for \; any} \; t \in [0, T'] \; {\it and} \; \tilde{u}_i \in K(M, T'),
\end{align}
where $u_i = \Gamma_{T'}(\tilde{u}_i)$ for $i = 1, 2$, $M$ and $T'$ are positive constants defined in Lemma {\rm \ref{KTM}}.
\end{lemma}
\begin{proof}
Let $v_1 = \hat{\lambda}(u_1), v_2 = \hat{\lambda}(u_2)$. 
By \eqref{sub-est} we get
\begin{align}
&\psi'(u_1)(u_{1t} - u_{2t}) + (\psi'(u_1) - \psi'(u_2))u_{2t} \nonumber\\
&= (v_1 - v_2)_{xx}
+ \lambda'(\tilde{u}_1)\tilde{u}_{1x}p_x
+ \lambda(\tilde{u_1}) p_{xx}
- \lambda'(\tilde{u}_2)\tilde{u}_{2x}p_x
- \lambda(\tilde{u_2}) p_{xx}  \nonumber\\
&= (v_1 - v_2)_{xx}
+ \lambda'(\tilde{u}_1) (\tilde{u}_1 - \tilde{u}_2)_x p_x
+ (\lambda'(\tilde{u}_1) - \lambda'(\tilde{u}_2)) \tilde{u}_{2x} p_x  \nonumber\\
&\label{laslem-eq0}\quad +(\lambda(\tilde{u}_1) - \lambda(\tilde{u}_2)) p_{xx} \;\; {\rm in} \; Q(T).
\end{align}
Here, for simplicity we put $v = v_1 - v_2, u = u_1 - u_2$ and $\tilde{u} = \tilde{u}_1 - \tilde{u}_2$.

First, we show that the function $t \;\mapsto \; \frac{1}{2} |v_x(t)|_H^2$ is differentiable a.e. on $[0, T]$ and
\begin{align}\label{laslem-eq1}
\frac{1}{2} |v_x(t)|_H^2 - \frac{1}{2} |v_x(s)|_H^2
\leq \int^t_s \frac{d}{d\tau} \left( \frac{1}{2} |v_x(t)|_H^2 \right) d\tau \;\; {\rm for} \; 0 \leq s \leq t \leq T.
\end{align}
In order to prove \eqref{laslem-eq1}, we put $v(t) = 0$ and $u(t) = 0$ for $t < 0$. Let $\Delta t > 0$ and multiply \eqref{laslem-eq0} by $\frac{v(t) - v(t - \Delta t)}{\Delta t}$, then we see that
\begin{align}\label{laslem-eq2}
&\int^1_0 \psi'(u_1(t))u_t(t)\frac{v(t) - v(t - \Delta t)}{\Delta t} dx - \int^1_0 v_{xx}(t) \frac{v(t) - v(t - \Delta t)}{\Delta t} dx \nonumber\\
&= \int^1_0 \lambda'(\tilde{u}_1(t)) \tilde{u}_x(t) p_x(t) \frac{v(t) - v(t - \Delta t)}{\Delta t} dx \nonumber\\
&\quad + \int^1_0 (\lambda'(\tilde{u}_1(t)) - \lambda'(\tilde{u}_2(t))) \tilde{u}_{2x} p_x \frac{v(t) - v(t - \Delta t)}{\Delta t} dx \nonumber\\
&\quad +\int^1_0(\lambda(\tilde{u}_1(t)) - \lambda(\tilde{u}_2(t))) p_{xx}(t) \frac{v(t) - v(t - \Delta t)}{\Delta t} dx \nonumber\\
&\quad - \int^1_0 (\psi'(u_1(t)) - \psi'(u_2(t)))u_{2t}(t) \frac{v(t) - v(t - \Delta t)}{\Delta t} dx \nonumber\\
& =: \sum^{12}_{i = 9} I_i(t) \;\; {\rm for \; a.e.} \; t \in [0, T].
\end{align}
By elementary calculations and the definition of $K(T', M)$, we have
\begin{align}
\label{laslem-eq3}I_9(t)
&\leq C_\lambda |p_x|_{L^\infty(Q(T))} |\tilde{u}_x(t)|_H \left| \frac{v(t) - v(t - \Delta t)}{\Delta t} \right|_H,\\
\label{laslem-eq4}I_{10}(t)
&\leq C_\lambda |p_x|_{L^\infty(Q(T))} |\tilde{u}(t)|_{L^\infty(0,1)} |\tilde{u}_{2x}(t)|_H \left| \frac{v(t) - v(t - \Delta t)}{\Delta t} \right|_H \nonumber\\
&\leq C_\lambda \sqrt{M} |p_x|_{L^\infty(Q(T))} |\tilde{u}(t)|_{L^\infty(0,1)} \left| \frac{v(t) - v(t - \Delta t)}{\Delta t} \right|_H,\\
\label{laslem-eq5}I_{11}(t)
&\leq C_\lambda |\tilde{u}(t)|_{L^\infty(0,1)} |p_{xx}(t)|_H \left| \frac{v(t) - v(t - \Delta t)}{\Delta t} \right|_H \nonumber\\
&\leq C_\lambda |p_{xx}|_{L^\infty(0,T; H)} |\tilde{u}(t)|_{L^\infty(0,1)} \left| \frac{v(t) - v(t - \Delta t)}{\Delta t} \right|_H,\\
\label{laslem-eq6}I_{12}(t)
&\leq C_\psi |u(t)|_{L^\infty(0,1)} |u_{2t}(t)|_H \left| \frac{v(t) - v(t - \Delta t)}{\Delta t} \right|_H \;\; {\rm for \; a.e.} \; t \in [0, T].
\end{align}
Also, we see that
\begin{align}\label{laslem-eq7}
&-\int^1_0 v_{xx}(t) \frac{v(t) - v(t - \Delta t)}{\Delta t} dx \nonumber\\
&= \int^1_0 \frac{(v_x(t))^2 - v_x(t)v_x(t - \Delta t)}{\Delta t} dx \nonumber\\
&\geq \frac{1}{2\Delta t}\left( \int^1_0 |v_x(t)|^2 dx - \int^1_0 |v_x(t - \Delta t)|^2 dx \right) \;\; {\rm for} \; 0 \leq t \leq T.
\end{align}
From \eqref{laslem-eq7}, these inequalities and (A1), we have
\begin{align}\label{laslem-eq8}
&\frac{1}{2\Delta t}\left( \int^1_0 |v_x(t)|^2 dx - \int^1_0 |v_x(t - \Delta t)|^2 dx \right) \nonumber\\
&\leq \hat{C}_8 (|\tilde{u}_x(t)|_H + |\tilde{u}(t)|_{L^\infty(0,1)} + |u(t)|_{L^\infty(0,1)}|u_{2t}(t)|_H + |u_t(t)|_H)\left| \frac{v(t) - v(t - \Delta t)}{\Delta t} \right|_H \\
&\hspace*{130mm}{\rm for} \; 0 \leq t \leq T, \nonumber
\end{align}
where $\hat{C}_8 = C_\lambda|p_x|_{L^\infty(Q(T))} + C_\lambda \sqrt{M} |p_x|_{L^\infty(Q(T))} + C_\lambda|p_{xx}|_{L^\infty(0,T; H)} + C_\psi$.

By integrating it with respect to $t$ over $[0, t_1]$ for $0 \leq t_1 \leq T$, we obtain
\begin{align*}
&\frac{1}{2 \Delta t} \left( \int^{t'}_0 \int^1_0 |v_x(t)|^2 dx dt - \int^{t'}_0 \int^1_0 |v_x(t - \Delta t)|^2 dx dt\right)\\
&\leq \hat{C}_8 \int^{t'}_0 F_3(\tau) \left| \frac{v(\tau) - v(\tau - \Delta t)}{\Delta t} \right|_H d\tau \;\; {\rm for} \; 0 \leq t' \leq T,
\end{align*}
where
\[
F_3(t) = |\tilde{u}_x(t)|_H + |\tilde{u}(t)|_{L^\infty(0,1)} + |u(t)|_{L^\infty(0,1)}|u_{2t}(t)|_H + |u_t(t)|_H \;\; {\rm for} \; 0 \leq t \leq T.
\]
Because of $\tilde{u} \in K(T', M)$,  Lemmas \ref{lem42} and \ref{lem43}, we have $\tilde{u}, u \in L^\infty(Q(T))$. This implies that $F_3 \in L^2(0, T)$. Also, it is clear that
\begin{align}\label{laslem-eq9}
\frac{v - v(\cdot - \Delta t)}{\Delta t} \to v_t \;\; {\rm in} \; L^2(0,T; H) \; {\rm as} \; \Delta t \to 0.
\end{align}
Moreover, we infer that
\begin{align*}
&\frac{1}{2 \Delta t} \left( \int^{t'}_0 \int^1_0 |v_x(t)|^2 dx dt - \int^{t'}_0 \int^1_0 |v_x(t - \Delta t)|^2 dx dt \right)\\
&= \frac{1}{2 \Delta t} \int^{t'}_{t' - \Delta t} |v_x(t)|_H^2 dt - \frac{1}{2 \Delta t} \int^0_{- \Delta t} |v_x(t)|_H^2 dt\\
&= \frac{1}{2 \Delta t} \int^{t'}_{t' - \Delta t} |v_x(t)|_H^2 dt - \frac{1}{2} |v_x(0)|_H^2 \;\; {\rm for} \; 0 \leq t' \leq T.
\end{align*}
Hence, we obtain
\begin{align}\label{laslem-eq10}
&\frac{1}{2 \Delta t} \int^{t'}_{t' - \Delta t} |v_x(t)|_H^2 dt - \frac{1}{2} |v_x(0)|_H^2 \nonumber \\
&\leq \hat{C}_8 \int^{t'}_0 |F_3(\tau)| \left| \frac{v(\tau) - v(\tau - \Delta t)}{\Delta t} \right|_H d\tau \;\; {\rm for} \; 0 \leq t' \leq T, \Delta t > 0.
\end{align}
Here, by applying the Lebesgue density point theorem, there exists a measurable set $E \subset [0,T]$ satisfying ${\rm mes}(E) = 0$ and
\begin{align}\label{laslem-eq11}
\lim_{\Delta t \to 0} \frac{1}{\Delta t} \int^t_{t - \Delta t} |v_x(\tau)|_H^2 d\tau = |v_x(t)|_H^2 \;\; {\rm for} \; t \in [0, T] \backslash E,
\end{align}
where ${\rm mes}(E)$ indicates the Lebesgue measure of $E$ in $\mathbb{R}$.

Thanks to \eqref{laslem-eq9} and \eqref{laslem-eq11}, by letting $\Delta t \to 0$ in \eqref{laslem-eq10} we get
\begin{align}\label{laslem-eq12}
\frac{1}{2} |v_x(t')|_H^2 - \frac{1}{2}|v_x(0)|_H^2
\leq \hat{C}_8 \int^{t'}_0 |F_3| |v_t|_H dt \;\; {\rm for} \; t' \in [0, T] \backslash E.
\end{align}
Now, we show that \eqref{laslem-eq12} hold for any $t' \in [0, T]$. Indeed, for any $t' \in [0, T]$ we can take a sequence $\{ t_n \} \subset [0, T] \backslash E$ such that $t_n \to t$ in $[0, T]$ as $n \to \infty$. Since $v \in W^{1,2}(0,T; H) \subset C([0,T]; H)$, \eqref{laslem-eq12} guarantees that the sequence $\{ v(t_n) \}$ is bounded in $X$. Hence, on account of $v \in C([0,T]; H)$, we obtain $v(t_n) \to v(t)$ weakly in $X$ as $n \to \infty$, namely, \eqref{laslem-eq12} holds for any $t \in [0, T]$.

For any $s \in [0, T]$, by regarding $s$ as the initial time, we can show the following inequality in a similar way:
\begin{align}\label{laslem-eq13}
\frac{1}{2} |v_x(t)|_H^2 - \frac{1}{2}|v_x(s)|_H^2 
\leq \int^t_s |F_3| |v_t|_H d\tau \;\; {\rm for} \; 0 \leq s < t \leq T.
\end{align}
This implies that the function $t \;\mapsto\; \frac{1}{2} |v_x(t)|_H^2$ is differentiable for a.e. $t \in [0, T]$ and it holds that
\[
\frac{1}{2} |v_x(t)|_H^2 - \frac{1}{2}|v_x(s)|_H^2
\leq \int^t_s \frac{d}{d\tau} \left( \frac{1}{2} |v_x(\tau)|^2 \right) d\tau \;\; {\rm for} \; 0 \leq s < t \leq T.
\]
Thus, we have proved \eqref{laslem-eq1}.

From now on we shall prove \eqref{diff-eq2}. By \eqref{laslem-eq2} and \eqref{laslem-eq7}, we have
\begin{align}\label{laslem-eq14}
&\int^1_0 \psi'(u_1(t))u_t(t)\frac{v(t) - v(t - \Delta t)}{\Delta t} dx
+ \frac{1}{\Delta t} \left( \frac{1}{2} |v_x(t)|_H^2 - \frac{1}{2} |v_x(t - \Delta t)|_H^2 \right) \nonumber\\
&\leq \int^1_0 \lambda'(\tilde{u}_1(t)) \tilde{u}_x(t) p_x(t) \frac{v(t) - v(t - \Delta t)}{\Delta t} dx \nonumber \\
&\quad + \int^1_0 (\lambda'(\tilde{u}_1(t)) - \lambda'(\tilde{u}_2(t)))\tilde{u}_{2x}(t) p_x(t) \frac{v(t) - v(t - \Delta t)}{\Delta t} dx \nonumber \\
&\quad + \int^1_0 (\lambda(\tilde{u}_1(t)) - \lambda(\tilde{u}_2(t)))p_{xx}(t) \frac{v(t) - v(t - \Delta t)}{\Delta t} dx \nonumber \\
&\quad - \int^1_0 (\psi'(u_1(t)) - \psi'(u_2(t)))u_{2t}(t) \frac{v(t) - v(t - \Delta t)}{\Delta t} dx \;\; {\rm for \; a.e.} \; t \in [0, T] \; {\rm and} \; \Delta t > 0.
\end{align}
Since the function $t \;\mapsto\; \frac{1}{2}|v_x(t)|_H^2$ and $v$ are differentiable a.e. on $[0, T]$, by letting $\Delta t \to 0$ in \eqref{laslem-eq14} we obtain:
\begin{align*}
&\int^1_0 \psi'(u_1) u_t v_t dx
+ \frac{d}{dt} \left( \frac{1}{2} |v_x|_H^2 \right)\\
&\leq \int^1_0 \lambda'(\tilde{u}_1) \tilde{u}_x p_x v_t dx
+ \int^1_0 (\lambda'(\tilde{u}_1) - \lambda'(\tilde{u}_2)) \tilde{u}_{2x} p_x v_t dx\\
&\quad + \int^1_0 (\lambda(\tilde{u}_1) - \lambda(\tilde{u}_2)) p_{xx} v_t dx
- \int^1_0 (\psi'(u_1) - \psi'(u_2))u_{2t} v_t dx \;\; {\rm for \; a.e.} \; t \in [0, T].
\end{align*}
Here, we note that $\frac{1}{C_\lambda} \leq (\hat{\lambda}^{-1}(r))' \leq \frac{1}{\delta_\lambda}$ for any $r \in \mathbb{R}$,  
\begin{align*}
u_t v_t
&= (\hat{\lambda}^{-1}(v_1) - \hat{\lambda}^{-1}(v_2))_t v_t\\
&= ((\hat{\lambda}^{-1})'(v_1) v_{1t} - (\hat{\lambda}^{-1})'(v_2) v_{2t}) v_t\\
&\geq \frac{1}{C_\lambda} |v_t|^2 - |((\hat{\lambda}^{-1})'(v_1) - (\hat{\lambda}^{-1})'(v_2)) v_{2t}| |v_t|\\
&\geq \frac{1}{C_\lambda} |v_t|^2 - \frac{C_\lambda^2}{\delta_\lambda^2} |v| |v_{2t}| |v_t|,
\end{align*}
and
\begin{align*}
|u|
&\leq |\hat{\lambda}^{-1}(v_1) - \hat{\lambda}^{-1}(v_2)|\\
&\leq \frac{1}{\delta_\lambda} |v| \;\; {\rm a.e. \; on} \; Q(T).
\end{align*}
Similarly to \eqref{laslem-eq3} - \eqref{laslem-eq6} from the above inequalities it follows that
\begin{align*}
\frac{\delta_\psi}{C_\lambda} |v_t|_H^2 + \frac{d}{dt} \left( \frac{1}{2} |v_x|_H^2 \right)
&\leq \frac{C_\lambda^2}{\delta_\lambda^2} |v|_{L^\infty(0, 1)} |v_{2t}|_H |v_t|_H
+ C_\lambda |p_x|_{L^\infty(Q(T))} |\tilde{u}_x|_H |v_t|_H\\
&\quad + C_\lambda |\tilde{u}|_{L^\infty(0, 1)} |p_x|_{L^\infty(Q(T))} |\tilde{u}_{xx}|_H |v_t|_H
+ C_\lambda |\tilde{u}|_{L^\infty(0, 1)} |p_{xx}|_H |v_t|_H\\
&\quad + \frac{C_\psi}{\delta_\lambda} |v|_{L^\infty(0, 1)} |u_{2t}|_H |v_t|_H \;\; {\rm a.e. \; on} \; [0, T].
\end{align*}
Thanks to Young's inequality and Lemma \ref{abs_ineq}, there exists a positive constant $\hat{C}_9$ such that
\begin{align}\label{laslem-eq15}
\frac{\delta_\psi}{2 C_\lambda} |v_t|_H^2
+ \frac{d}{dt} \left( \frac{1}{2} |v_x|_H^2 \right)
\leq \hat{C}_9 (|v|_H^2 + |v_x|_H^2) (|v_{2t}|_H^2 + |u_{2t}|_H^2)
+ \hat{C}_9 |\tilde{u}|_X^2 \;\; {\rm for \; a.e.} \; t \in [0, T].
\end{align}
Note that by \eqref{laslem-eq13} the function $t \; \mapsto \; \frac{1}{2} e^{-\int^t_0 F_4 d\tau}|v_x(t)|_H^2$ is differentiable a.e. on $[0, T]$ and
\[
\frac{1}{2} e^{-\int^t_0 F_4 d\tau}|v_x(t)|_H^2
-\frac{1}{2} e^{-\int^s_0 F_4 d\tau}|v_x(s)|_H^2
\leq \frac{1}{2} \int^t_s \frac{d}{dt} \left(e^{-\int^t_0 F_4 d\tau}|v_x(\tau)|_H^2 \right) d\tau \;\; {\rm for} \; 0 \leq s \leq t \leq T,
\] 
where $F_4 = \hat{C}_9 (|v_{2t}|_H^2 + |u_{2t}|_H^2)$ on $[0, T]$.

Therefore, we can apply Gronwall's inequality to \eqref{laslem-eq15} and have obtained
\begin{align}\label{laslem-eq16}
\frac{\delta_\psi}{2C_\lambda} \int^t_0 |v_t|_H^2 d\tau 
+ \frac{1}{2} |v_x(t)|_H^2
\leq \hat{C}_9 e^{\int^t_0 F_4 d\tau}\left(\int^t_0 F_4 |u|_H^2 d\tau + \int^t_0 |\tilde{u}|_X^2 d\tau \right) \;\; {\rm for} \; 0 \leq t \leq T.
\end{align}
Moreover, we see that
\begin{align}\label{laslem-eq17}
\int^t_0 F_4 |u|_H^2 d\tau
\leq \hat{C}_{10} \left(\int^t_0 |\tilde{u}|_H^2 d\tau \right) \left(\int^t_0 F_4 d\tau \right) \;\; {\rm for} \;\; 0 \leq t \leq T',
\end{align}
where $\hat{C}_{10}$ is a positive constant depending on $C_5$ given in Lemma \ref{lem45}.

Due to \eqref{laslem-eq16} and \eqref{laslem-eq17}, there exists a positive constant $\hat{C}_{11}$ such that
\begin{align*}
|(v_1 - v_2)_x(t')|_H^2 \leq \hat{C}_{11} \int^{t'}_{0} |\tilde{u}_1 - \tilde{u}_2|^2_Xdt \;\; {\rm for \; any } \; t' \in [0, T'].
\end{align*}
Here, by Lemma \ref{lem43}, $|(u_1 - u_2)_x(t')|_H^2$ is estimated as follows: 
\begin{align*}
|(u_1 - u_2)_x(t')|_H^2
&\leq |(v_1 - v_2)_x(t')|_H^2 + \frac{12C_\lambda^2}{\delta_\lambda^6} |(v_1 - v_2)(t')|^2_X |v_{2x}|^2_H\\
&\leq \left( \hat{C}_{11} + \frac{12C_\lambda^4C_2(1 + MT')}{\delta_\lambda^6} \right) \int^{t'}_{0} |\tilde{u}_1 - \tilde{u}_2|^2_Xdt\\
&=: C_6 \int^{t'}_{0} |\tilde{u}_1 - \tilde{u}_2|^2_Xdt \;\; {\rm for \; any } \; t' \in [0, T'].
\end{align*}
Thus, we have proved this lemma.
\end{proof}

\begin{proof}[Proof of Theorem \ref{P-SS}]
For $T_1 \in [0, T']$ by Lemmas \ref{lem45} and \ref{lem46}, we have 
\begin{align*}
\int^{T_1}_0 |(u_1 - u_2)(t)|_X^2 dt
&\leq (C_5 + C_6)T_1 \int^{T_1}_0 |(\tilde{u}_1 - \tilde{u}_2)(t)|^2_X  dt.
\end{align*}

Thus, we obtain
\begin{align*}
|u_1 - u_2|_{L^2(0, T_1; X)}
 &\leq \sqrt{(C_5 + C_6)T_1}|\tilde{u}_1 - \tilde{u}_2|_{L^2(0, T_1; X)}.
\end{align*}

Here we choose $T_1 \in (0, T']$ such that $\sqrt{(C_5 + C_6)T_1} < 1$. Accordingly, for any $\tilde{u}_i \in K(T_1, M), i = 1, 2$, let $u_1 = \Gamma_{T_1}(\tilde{u}_1), u_2 = \Gamma_{T_1}(\tilde{u}_2)$ and put $\mu = \sqrt{(C_5 + C_6)T_1}$. Then, it holds that
\[
|u_1 - u_2|_{L^2(0, T_1; W^{1,2}(0, 1))}
 \leq \mu|\tilde{u}_1 - \tilde{u}_2|_{L^2(0, T_1; W^{1,2}(0, 1))}.
 \]
 
Hence, the Banach fixed point theorem implies that there exists one and only one $u \in K(T_1, M)$ satisfying $\Gamma_{T_1}(u) = u$. Therefore, (P) has a unique solution $u$ on $[0, T_1]$.\\

Finally, we show that (P) has a solution globally in time. First, we note that by Lemma \ref{lem42}, \ref{lem43}, we can estimate $|u(T_1)|_X$ only by $T$. Accordingly, we can solve (P) on $[T_1, T_2]$ by regarding $u(T_1)$ as the initial function. This shows that (P) has a solution on $[0, T_2]$. Since $T_2 - T_1$ depends only on $T$, by repeating this argument finite times, we can obtain the solution of (P) on $[0, T]$. Thus, Theorem \ref{P-SS} is proved, completely.
\end{proof}

\begin{center}
{\bf Acknowledgement }
\end{center}
 
The authors are grateful to members of Ebara corporation for showing several interesting phenomena related to moisture penetration with fruitful discussion which triggered this work.

\bibliographystyle{plain}

\end{document}